\documentclass[11pt]{article}
\usepackage{amsthm, amsmath, amssymb, amsfonts, url, booktabs, float, tikz, setspace, fancyhdr, verbatim, chngcntr}
\usepackage[margin = 1in]{geometry}
\usepackage{enumitem}

\usepackage{hyperref}
\usepackage{tikz}

\counterwithin{figure}{section}

\newtheorem{theorem}{Theorem}[section]
\newtheorem{proposition}[theorem]{Proposition}
\newtheorem{lemma}[theorem]{Lemma}
\newtheorem{corollary}[theorem]{Corollary}

\newtheorem*{thm}{Theorem}

\theoremstyle{definition}
\newtheorem{definition}[theorem]{Definition}

\theoremstyle{remark}

\newcommand{\abs}[1]{\left\lvert#1\right\rvert}

\newcommand{\rprod}[1]{\ominus_{#1}}
\newcommand{\rrprod}[1]{\odot_{#1}}

\newcommand{\PP}{\mathbb{P}}

\newcommand{\mc}[1]{\mathcal{#1}}

\newcommand{\eps}{\varepsilon}

\newcommand{\fred}{F_{red}}
\newcommand{\fblue}{F_{blue}}
\newcommand{\gred}{G_{red}}
\newcommand{\gblue}{G_{blue}}
\newcommand{\ggreen}{G_{green}}
\newcommand{\G}{\Gamma}
\newcommand{\Gred}{\Gamma_{red}}
\newcommand{\Gblue}{\Gamma_{blue}}
\newcommand{\Ggreen}{\Gamma_{green}}
\newcommand{\mred}{M_{red}}
\newcommand{\mblue}{M_{blue}}

\tikzstyle{p}+=[fill=black, circle, minimum width = 1pt, inner sep =
1pt]
\tikzstyle{w}+=[fill=white, draw, circle, minimum width = 1pt, inner sep =
1.5pt]

\begin{document}

\title{Ramsey games near the critical threshold}

\author{David Conlon\thanks{Department of Mathematics, California Institute of Technology, Pasadena, CA 91125, USA.
E-mail: {\tt
dconlon@caltech.edu}. Research supported in part by ERC Starting Grant RanDM 676632.} \and 
Shagnik Das\thanks{Institut f\"ur Mathematik, Freie Universit\"at Berlin, 14195 Berlin, Germany.  E-mail: {\tt shagnik@mi.fu-berlin.de}.  Research supported by GIF grant G-1347-304.6/2016 and by the Deutsche Forschungsgemeinschaft (DFG) project 415310276.}\and
Joonkyung Lee\thanks{
Department of Mathematics, University College London, Gower Street, London WC1E 6BT, UK.
E-mail: {\tt
joonkyung.lee@ucl.ac.uk}. Research supported in part by ERC Consolidator Grant PEPCo 724903.}\and
Tam\'as M\'esz\'aros\thanks{Institut f\"ur Mathematik, Freie Universit\"at Berlin, 14195 Berlin, Germany.  E-mail: {\tt tamas.meszaros@fu-berlin.de}.  Funded by the DRS Fellowship Program and the  Berlin Mathematics Research Center MATH+, Project ``Learning hypergraphs".}
}

\date{}

\maketitle

\begin{abstract}
A well-known result of R\"odl and Ruci\'nski states that for any graph $H$ there exists a constant $C$ such that if $p \geq C n^{- 1/m_2(H)}$, then the random graph $G_{n,p}$ is a.a.s.~$H$-Ramsey, that is, any $2$-colouring of its edges contains a monochromatic copy of $H$. Aside from a few simple exceptions, the corresponding $0$-statement also holds, that is, there exists $c>0$ such that whenever $p\leq cn^{-1/m_2(H)}$ the random graph $G_{n,p}$ is a.a.s.~not $H$-Ramsey.

We show that near this threshold, even when $G_{n,p}$ is not $H$-Ramsey, it is often extremely close to being $H$-Ramsey.  More precisely, we prove that for any constant $c > 0$ and any strictly $2$-balanced graph $H$, if $p \geq c n^{-1/m_2(H)}$, then the random graph $G_{n,p}$ a.a.s.~has the property that every $2$-edge-colouring without monochromatic copies of $H$ cannot be extended to an $H$-free colouring after $\omega(1)$ extra random edges are added. This generalises a result by Friedgut, Kohayakawa, R\"odl, Ruci\'nski and Tetali, who in 2002 proved the same statement for triangles, and addresses a question raised by those authors.  We also extend a result of theirs on the three-colour case and show that these theorems need not hold when $H$ is not strictly $2$-balanced.
\end{abstract}

\section{Introduction}

The study of sparse generalisations of combinatorial theorems has attracted considerable interest in recent years and there are now several general mechanisms~\cite{BMS15, CG16, ST15, S16} that allow one to prove that analogues of classical results such as Ramsey's theorem, Tur\'an's theorem and Szemer\'edi's theorem hold relative to sparse random graphs and sets of integers.
Much of this work is based, in one way or another, on the beautiful random Ramsey theorem of R\"odl and Ruci\'nski~\cite{RR93, RR95} from 1995. This seminal result gives a complete answer to the question of when the binomial random graph $G_{n,p}$ is \emph{$(H, r)$-Ramsey}, that is, has the property that any $r$-colouring of its edges contains a monochromatic copy of the graph $H$.

To state the R\"odl--Ruci\'nski theorem precisely, we need some notation. For a graph $H$, we write $d_2(H) = 0$ if $H$ has no edges, $d_2(H) = 1/2$ when $H = K_2$ and $d_2(H) = (e(H)-1)/(v(H)-2)$ in the general case. We then write $m_2(H) = \max_{H' \subseteq H} d_2(H')$ and call this quantity the \emph{$2$-density} of $H$. Though we will not use these definitions immediately, we also say that $H$ is \emph{$2$-balanced} if $m_2(H') \leq m_2(H)$ and \emph{strictly $2$-balanced} if $m_2(H') < m_2(H)$ for all proper subgraphs $H'$ of $H$.

\begin{thm}(R\"odl--Ruci\'nski, 1995) \label{thm:rodlrucinski}
Let $r \geq 2$ be a positive integer and let $H$ be a graph that is not a forest consisting of stars and paths of length $3$. Then there are positive constants $c$ and $C$ such that
\[\lim_{n\rightarrow \infty} \PP[G_{n,p} \mbox{ is $(H, r)$-Ramsey}] = \left\{ \begin{array}{ll}
         0 & \mbox{if $p < c n^{-1/m_2(H)}$},\\
         1 & \mbox{if $p > C n^{-1/m_2(H)}$}.\end{array} \right.\]
\end{thm}

There has been much work extending this result. We will not attempt an exhaustive survey, but refer the interested reader instead to some of the latest progress on hypergraphs~\cite{GNPSST17}, the asymmetric case~\cite{MNS19}, establishing sharp thresholds~\cite{SS18} and the equivalent problem in settings other than the binomial random graph~\cite{DT19, P19}. Our particular concern here will be with the following surprising result of Friedgut, Kohayakawa, R\"odl, Ruci\'nski and Tetali~\cite{FKRRT03} regarding two-round Ramsey games against a random builder. 

\begin{thm}(Friedgut--Kohayakawa--R\"odl--Ruci\'nski--Tetali, 2003)
Let $c > 0$ be fixed and, for $p = cn^{-1/2}$, let $G = G_{n,p}$. Then, with high probability, the following statements hold:
\begin{itemize}
	\item[(a)] Let $\varphi_2$ be an arbitrary monochromatic-$K_3$-free $2$-edge-colouring of $G$.  If $q_2 = \omega(n^{-2})$, then, with high probability, $\varphi_2$ cannot be extended to a monochromatic-$K_3$-free $2$-edge-colouring of $G \cup G_{n,q_2}$.
	\item[(b)] Let $\varphi_3$ be an arbitrary monochromatic-$K_3$-free $3$-edge-colouring of $G$.  If $q_3 = \omega(n^{-1})$, then, with high probability, $\varphi_3$ cannot be extended to a monochromatic-$K_3$-free $3$-edge-colouring of $G \cup G_{n,q_3}$.
\end{itemize}
\end{thm}

When $H = K_3$, the R\"odl--Ruci\'nski theorem implies that if $p = C n^{-1/2}$ for some sufficiently large $C$, then every $2$-edge-colouring contains a monochromatic triangle. Part (a) of the theorem above says that for any $c > 0$, no matter how small, if $p = c n^{-1/2}$, then, even though there are $2$-edge-colourings of $G_{n,p}$ containing no monochromatic $K_3$, no such colouring can be extended to a monochromatic-$K_3$-free $2$-edge-colouring after $\omega(1)$ extra random edges are added. One interpretation of this result is that for any $c > 0$ the random graph $G_{n,p}$ with $p = c n^{-1/2}$ is, with high probability, already extremely close to being $(K_3, 2)$-Ramsey. Part (b) gives a similar result for $3$-edge-colourings, though in this case $\omega(n)$ extra edges may be needed in the second round of colouring to guarantee a monochromatic triangle.

Addressing a problem raised by Friedgut, Kohayakawa, R\"odl, Ruci\'nski and Tetali~\cite{FKRRT03}, our main result says that a similar statement holds for all graphs $H$ containing an edge $h$ for which $m_2(H \setminus h) < m_2(H)$. In particular, the result applies when $H$ is strictly $2$-balanced, since any edge $h$ works in this case.

\begin{theorem} \label{thm:main}
Let $H$ be a graph and suppose that there is some edge $h \in E(H)$ whose removal decreases the $2$-density, that is, $m_2(H \setminus h) < m_2(H)$.  Let $c > 0$ be fixed and, for $p = c n^{-1/m_2(H)}$, let $G = G_{n,p}$.  Then, with high probability, the following statements hold:
\begin{itemize}
	\item[(a)] Let $\varphi_2$ be an arbitrary monochromatic-$H$-free $2$-edge-colouring of $G$.  If $q_2 = \omega(n^{-2})$, then, with high probability, $\varphi_2$ cannot be extended to a monochromatic-$H$-free $2$-edge-colouring of $G \cup G_{n,q_2}$.
	\item[(b)] Let $\varphi_3$ be an arbitrary monochromatic-$H$-free $3$-edge-colouring of $G$.  If $q_3 = \omega(n^{-1/m(H)})$, then, with high probability, $\varphi_3$ cannot be extended to a monochromatic-$H$-free $3$-edge-colouring of $G \cup G_{n,q_3}$.
\end{itemize}
\end{theorem}

Observe that the densities $q_2$ and $q_3$ of the random graphs that must be added to create monochromatic copies of $H$ are best possible. Indeed, if $q = O(n^{-2})$, then with positive probability $G_{n,q}$ has no edges, so $\varphi_2$ trivially extends to $G \cup G_{n,q}$. Only slightly less trivially, if $\varphi_3$ only uses two of the three colours on the edges of $G$, then we can colour all the edges of $G_{n,q}$ with the third colour.  If $q = O(n^{-1/m(H)})$, then with positive probability $G_{n,q}$ is $H$-free, thus giving a valid extension of $\varphi_3$. Finally, note that these results cannot be extended to $r \ge 4$ colours, since the two random graphs $G_{n,p}$ and $G_{n,q}$ can be coloured independently with disjoint pairs of colours, so we can avoid creating a monochromatic copy of $H$ until the density of one of the two random graphs exceeds the random Ramsey threshold $C n^{-1/m_2(H)}$ from Theorem~\ref{thm:rodlrucinski}.

\section{The necessity of a condition}

In Theorem~\ref{thm:main}, we impose the condition that there is some edge $h \in E(H)$ such that $H \setminus h$ has a strictly lower $2$-density than $H$.  While this condition covers, for example, strictly $2$-balanced graphs (where the edge $h$ can be chosen arbitrarily), it is natural to ask whether it is necessary.  In this section we show that Theorem~\ref{thm:main} does not apply to all graphs $H$, so some condition is indeed required.

\subsection{Edge-rooted products of graphs}

We first define the edge-rooted product of graphs.

\begin{definition} \label{def:product}
Let $G$ be a graph, let $H$ be a graph rooted at an edge $h = \{u,v\} \in E(H)$ and let $k \in \mathbb{N}$.  To build the \emph{$k$-fold edge-rooted product} $G \rprod{k} (H,h)$, we start with a central copy of $G$ and then attach $k$ copies of $H$ to each edge $g = \{x,y\} \in E(G)$ such that $\{x,y\}$ is the root-edge $h$ in each copy of $H$ and all other vertices in each copy are new and distinct.

In other words, $V( G \rprod{k} (H,h) ) = V(G) \cup \left( E(G) \times [k] \times \left( V(H) \setminus \{u,v\} \right) \right)$, $V(G)$ induces a copy of $G$ and, for each $g = \{x,y\} \in E(G)$ and $i \in [k]$, $\left( \{g \} \times \{i \} \times \left( V(H) \setminus \{u,v\} \right) \right) \cup \{x,y\}$ induces a copy of $H$ with $\{x,y\}$ playing the role of $\{u,v\}$. (Note that there is some slack in this definition, since we have not prescribed an orientation for each attached copy of $H$. In practice, the particular choice of orientation makes no difference, so we will simply assume that some fixed choice has been made.)

The \emph{reduced} $k$-fold edge-rooted product, denoted $G \rrprod{k} (H,h)$, is the subgraph obtained by removing all the edges from the central copy of $G$.
\end{definition}

\begin{figure}[h!]
    \centering
    \begin{tikzpicture}[scale=1.5,every node/.style={circle,draw,color=black,fill=black,inner sep=0pt,minimum width=3pt}]
       \foreach \i in {1,2,3,4}
       {
       \begin{scope}[shift={(-4,0)}]
         \node (a\i) at (90*\i-15:1.5) {};
         \node (b\i) at (90*\i+15:1.5) {};
         \node (v\i) at (90*\i - 45:1) {};
        \end{scope}
        \begin{scope}[]shift={(4,0)}]
         \node (c\i) at (90*\i-15:1.5) {};
         \node (d\i) at (90*\i+15:1.5) {};
         \node (u\i) at (90*\i - 45:1) {};
        \end{scope}
        } 
       \foreach \i / \j in {1/2,2/3,3/4,4/1}
       {
       \draw (v\i)--(v\j);
       \draw (v\i)--(a\i)--(v\j);
       \draw (v\i)--(b\i)--(v\j);
       \draw (u\i)--(c\i)--(u\j);
       \draw (u\i)--(d\i)--(u\j);
       }
       
    \end{tikzpicture}
    \caption{$C_4\rprod{2}K_3$ (on the left) and $C_4\rrprod{2}K_3$ (on the right).}
    \label{fig:examples}
\end{figure}

We have already defined $d_2(H)$, $m_2(H)$ and stated what it means for a graph to be $2$-balanced or strictly $2$-balanced. In a similar fashion, we write $d_1(H) = 0$ if $H$ has no edges and $d_1(H) = e(H)/(v(H)-1)$ otherwise. We then write $m_1(H) = \max_{H' \subseteq H} d_1(H')$ and call this quantity the \emph{$1$-density} of $H$. We say that $H$ is \emph{$1$-balanced} if $m_1(H') \leq m_1(H)$ and \emph{strictly $1$-balanced} if $m_1(H') < m_1(H)$ for all proper subgraphs $H'$ of $H$. Finally, write $d(H) = e(H)/v(H)$ and $m(H) = \max_{H' \subseteq H} d(H')$, which we call the \emph{density} of $H$. We then say that $H$ is \emph{balanced} if $m(H') \leq m(H)$ and \emph{strictly balanced} if $m(H') < m(H)$ for all proper subgraphs $H'$ of $H$. We will make repeated use of the following simple lemma in what follows.

\begin{lemma} \label{lem:2become1}
If $H$ is $2$-balanced with $d_2(H) > 1$, then $H$ is strictly $1$-balanced and strictly balanced.
\end{lemma}

\begin{proof}
Suppose that $H$ is not strictly $1$-balanced and let $F \subset H$ be a subgraph with $d_1(F) \ge d_1(H)$. That is, $e(F) / (v(F) - 1) \geq e(H) / (v(H) - 1)$ or, by multiplying the expression out,
\begin{equation} \label{eq:1bal}
e(F) v(H) - e(F) \geq e(H) v(F) - e(H).
\end{equation} 
Since $H$ is $2$-balanced, we have $d_2(H) \ge d_2(F)$, which implies that
$(e(H) - 1)/(v(H) - 2) \geq (e(F) - 1)/(v(F) - 2)$. Rearranging gives
$e(H)v(F) - 2e(H) - v(F) + 2 \geq e(F)v(H) - 2e(F) - v(H) + 2$.
Substituting \eqref{eq:1bal}, we get
\[e(H)v(F) - 2e(H) - v(F) + 2 \geq e(F)v(H) - 2e(F) - v(H) + 2 \geq e(H)v(F) - e(H) - e(F) - v(H) + 2.\]
Cancelling the like terms gives $-e(H) - v(F) \geq -e(F) - v(H)$, which in turn implies that $(e(F) - 1) - (v(F) - 2) \geq (e(H) - 1) - (v(H) - 2)$, which can be rewritten as
$(d_2(F) - 1)(v(F) - 2) \geq (d_2(H) - 1)(v(H) - 2)$. 
However, this is a contradiction, since by assumption $d_2(H) - 1 \geq d_2(F) - 1$ and $v(H) - 2 > v(F) - 2$. The argument in the strictly balanced case follows along almost exactly the same lines.
\end{proof}

The key observation for our purposes is that the edge-rooted product behaves well with respect to the various graph densities. 

\begin{lemma} \label{lem:productdensity}
For any graphs $G$ and $H$ of density at least $1$, any edge $h \in E(H)$ and any $k \in \mathbb{N}$:
\begin{itemize}
	\item[(a)] if $G$ is strictly balanced, $H$ is $2$-balanced and $d(G) < d_2(H)$, then $G \rprod{k} (H,h)$ is strictly balanced,
	\item[(b)] $m_2( G \rprod{k} (H,h) ) = \max \{ m_2(G), m_2(H) \}$ and
	\item[(c)] if $m_2(H \setminus h) < m_2(H)$, then $m_2( G \rrprod{k} (H,h) ) < m_2(H)$.
\end{itemize}
\end{lemma}

\begin{proof}
\begin{itemize}
	\item[(a)]  Let $F \subseteq G \rprod{k} (H,h)$ be a smallest (induced) subgraph maximising $d(F) = e(F) / v(F)$. We wish to show that $F = G \rprod{k} (H,h)$.  We start with a lower bound on the density of $G \rprod{k} (H,h)$:
\begin{align} 
	m(G \rprod{k} (H,h)) &\ge d( G \rprod{k} (H,h) ) = \frac{e(G) + k e(G) (e(H) - 1)}{v(G) + k e(G) (v(H) - 2)} \notag \\
	&= \frac{e(H) - (1 - \tfrac{1}{k})}{v(H) - 1 - (1 - \tfrac{1}{kd(G)})} \ge \frac{e(H)}{v(H) - 1} = d_1(H) = m_1(H), \label{ineq:largerthan1densityb} 
\end{align}
where the inequality on the second line follows since either $k d(G) = 1$, in which case we have equality, or $(1 - \tfrac{1}{k})/(1 - \tfrac{1}{kd(G)}) \le 1 \le d(H) < d_1(H)$.  Note that the final equality, $d_1(H) = m_1(H)$, is an application of Lemma~\ref{lem:2become1} above.

We also observe that $d(G \rprod{k} (H,h))$ is a convex combination of $d(G)$ and $d_2(H)$:
\begin{equation} \label{ineq:largerthan1densitya}
    \frac{e(G) + k e(G) (e(H) - 1)}{v(G) + k e(G) (v(H) - 2)} = d(G) \frac{v(G)}{v(G \rprod{k} (H,h))} + d_2(H) \left(1 - \frac{v(G)}{v(G \rprod{k} (H,h))} \right).
\end{equation}

	Now, for each $g \in E(G)$ and $i \in [k]$, let $F_{g,i} \subseteq H$ be the subgraph induced by the vertices of $F$ in the $i$th copy of $H$ attached to the edge $g$ in the central copy of $G$.  Let $F_0 \subseteq G$ be the subgraph induced by the vertices of $F$ in the central copy of $G$.
	
	By the minimality of the size of $F$, we may assume that $F$ is connected, as otherwise its densest component would be a smaller subgraph attaining the maximum density. We cannot have $F \subseteq H$ since, by~\eqref{ineq:largerthan1densityb}, the density of $G \rprod{k} (H,h)$ is at least $m_1(H)$, which is strictly larger than $m(H)$.  Thus, $F_{g,i}$ must be non-empty for at least two pairs $(g,i) \in E(G) \times [k]$ and, hence, to be connected, each non-empty $F_{g,i}$ must contain at least one vertex of $g$.
	
	Now suppose there was some $(g,i)$ such that $F_{g,i}$ contained only one of the two endpoints of $g$.  Then, by removing $F_{g,i}$ from $F$, we lose $e(F_{g,i})$ edges and $v(F_{g,i})-1$ vertices.  Since $F_{g,i} \subset H$, the ratio $e(F_{g,i})/(v(F_{g,i})-1)$ is at most $m_1(H)$, which by~\eqref{ineq:largerthan1densityb} is at most $m(G \rprod{k} (H,h))$.  Removing $F_{g,i}$ would therefore not decrease the density of $F$, contradicting the minimality of its size.  Thus, if $F_{g,i}$ is non-empty, we must have $g \in E(F_{g,i})$. Hence,
	\begin{equation} \label{eqn:prodparameters}
	v(F) = v(F_0) + \sum_{g \in E(F_0)} \sum_{i = 1}^k \left( v(F_{g,i}) - 2 \right) \quad \textrm{and} \quad e(F) = e(F_0) + \sum_{g \in E(F_0)} \sum_{i=1}^k \left( e(F_{g,i}) - 1 \right).
	\end{equation}
	Thus,
	\[ d(F) = \frac{e(F_0) + \sum_{g \in E(F_0)} \sum_{i=1}^k (e(F_{g,i}) - 1)}{v(F_0) + \sum_{g \in E(F_0)} \sum_{i=1}^k (v(F_{g,i}) - 2)} = d(F_0) \frac{v(F_0)}{v(F)} + \sum_{g \in E(F_0)} \sum_{i=1}^k d_2(F_{g,i}) \frac{v(F_{g,i}) - 2}{v(F)}. \]
	
	Since $F_0 \subseteq G$ and $G$ is balanced, $d(F_0) \le d(G)$.  Similarly, for each $g$ and $i$, $d_2(F_{g,i}) \le d_2(H)$.  We therefore have
	\[ d(F) \le d(G) \frac{v(F_0)}{v(F)} + d_2(H) \left( 1 - \frac{v(F_0)}{v(F)} \right). \]
	
	Comparing this to~\eqref{ineq:largerthan1densitya}, since $d(G) < d_2(H)$, for $d(F) \ge d(G \rprod{k} (H,h))$ to hold we require 
	\begin{equation} \label{ineq:F0contribution}
		\frac{v(F_0)}{v(F)} \le \frac{v(G)}{v(G \rprod{k} (H,h))} = \frac{1}{1 + k d(G) (v(H) - 2)}.
	\end{equation}
	
	Now $v(F) = v(F_0) + \sum_{g \in E(F_0)} \sum_{i=1}^k (v(F_{g,i}) - 2) \le v(F_0) + k e(F_0) (v(H) - 2)$, with equality if and only if $F_{g,i} = H$ for all $g \in E(F_0)$ and $i \in [k]$.  Therefore,
	\[ \frac{v(F_0)}{v(F)} \ge \frac{v(F_0)}{v(F_0) + k e(F_0) (v(H) - 2)} = \frac{1}{1 + k d(F_0) (v(H) - 2)}. \]
	
	Thus, in order to satisfy the inequality of~\eqref{ineq:F0contribution}, $d(F_0) \ge d(G)$.  As $G$ is strictly balanced, it follows that $F_0 = G$ and then, since $F_{g,i} = H$ for all $g$ and $i$, we have $F = G \rprod{k} (H,h)$, as required.
	
	\item[(b)]  Since $G, H \subseteq G \rprod{k} (H,h)$, we immediately have $m_2(G \rprod{k} (H,h)) \ge \max \{ m_2(G), m_2(H) \}$.  The proof of the upper bound follows the same lines as in part (a).  Let $F \subseteq G \rprod{k} (H,h)$ be a smallest subgraph realising the $2$-density, that is, $m_2(G \rprod{k} (H,h)) = d_2(F) = (e(F) - 1)/(v(F) - 2)$.
	
	Let $F_0$ and, for each $(g,i) \in E(G) \times [k]$, $F_{g,i}$ be defined as in part (a).  We may assume that $F_{g,i} \neq \emptyset$ for at least two pairs $(g,i)$, since otherwise $F \subseteq H$ and thus $d_2(F) \le m_2(H)$.  By the minimality of the size of $F$, we may further assume that $F$ is $2$-connected, as otherwise one of the blocks $B$ of $F$ will satisfy $m_2(B) \ge m_2(F)$ (see, for instance, Lemma 8 of~\cite{NS16}).  In particular, this implies that $g \in E(F_{g,i})$ whenever $F_{g,i} \neq \emptyset$.
	
	The vertices and edges of $F$ can then be enumerated as in~\eqref{eqn:prodparameters}, so
	\begin{equation} \label{eqn:2densityofF}
	d_2(F) = \frac{e(F) - 1}{v(F) - 2} = \frac{e(F_0) - 1 + \sum_{g \in E(F_0)} \sum_{i = 1}^k \left( e(F_{g,i}) - 1 \right)}{v(F_0) - 2 + \sum_{g \in E(F_0)} \sum_{i=1}^k \left( v(F_{g,i}) - 2 \right)}.
	\end{equation}
	Since $e(F_0)-1 \leq m_2(G) (v(F_0) - 2)$ and $e(F_{g, i}) - 1 \leq m_2(H) (v(F_{g,i})-2)$ for each $(g,i)$, it follows that $d_2(F) = m_2(G \rprod{k} (H,h)) \le \max \{ m_2(G), m_2(H) \}$.
	
	\item[(c)]  The product $G \rrprod{k} (H,h)$ is obtained by deleting the edges of the central copy of $G$ from the product $G \rprod{k} (H,h)$.  We show $m_2(G \rrprod{k} (H,h)) < m_2(H)$ by following the argument of part (b).  To start, let $F \subseteq G \rrprod{k} (H,h)$ be a smallest subgraph attaining the $2$-density.
	
	As before, let $F_0$ be the subgraph of $G$ induced by the vertices of $F$ from the central copy of $G$ and, for $g \in E(G)$ and $i \in [k]$, let $F_{g,i}$ be the subgraph of $H$ induced by the vertices of $F$ in the $i$th copy of $H$ attached to the edge $g$.  Note that, since the edges from the central copy of $G$ are deleted in $G \rrprod{k} (H,h)$, neither the edges of $F_0$ nor the edge $g$ in $F_{g,i}$ (if present) appear in $F$.  However, it will be convenient for us to include them in $F_0$ and $F_{g,i}$ for our calculations.
	
	If $F_{g,i}$ is only non-empty for one pair of $(g,i)$, then $F = F_{g,i} \setminus g \subseteq H \setminus h$, so $d_2(F) \le m_2(H \setminus h) < m_2(H)$.  Otherwise, since $F$ must be $2$-connected, $g$ must be in $F_{g,i}$ whenever $F_{g,i}$ is non-empty.  We can then compute the $2$-density of $F$ as in part (b), arriving at an expression similar to~\eqref{eqn:2densityofF}, except the edges in $F_0$ do not appear in $F$.  Thus,
	\[ d_2(F) = \frac{-1 + \sum_{g \in E(F_0)} \sum_{i=1}^k \left( e(F_{g,i}) - 1 \right) }{v(F_0)- 2 + \sum_{g \in E(F_0)} \sum_{i=1}^k \left( v(F_{g,i}) - 2 \right) } < \frac{\sum_{g \in E(F_0)} \sum_{i=1}^k \left( e(F_{g,i}) - 1 \right)}{\sum_{g \in E(F_0)} \sum_{i=1}^k \left( v(F_{g,i}) - 2 \right)} \le m_2(H), \]
	since $F_{g,i} \subseteq H$ implies $e(F_{g,i}) - 1 \le m_2(H) (v(F_{g,i}) - 2)$. \qedhere
\end{itemize}
\end{proof}

\subsection{Graphs requiring unusually many extra random edges}

Part (c) of Lemma~\ref{lem:productdensity} shows the role played by the assumption of the existence of the edge $h$ in Theorem~\ref{thm:main}.  We will show how to use this to prove Theorem~\ref{thm:main} in the next section, but first we use the other parts of this lemma to construct graphs for which the conclusion of Theorem~\ref{thm:main} does not hold.

\begin{theorem} \label{thm:condition}
Let $F$ be a $2$-balanced graph containing a cycle, let $f \in E(F)$ be an arbitrary edge of $F$ and let $H = F \rprod{1} (F,f)$.  Let $G = G_{n,p}$ for $p = c n^{-1/m_2(H)}$, where $c > 0$ is a sufficiently small constant. Then, with high probability, the following statements hold:
\begin{itemize}
	\item[(a)] There is a monochromatic-$H$-free $2$-edge-colouring of $G$ such that if $q = o(n^{-v(F) / e(F)})$ the colouring can with high probability be extended to a colouring of $G \cup G_{n,q}$ without monochromatic copies of $H$.
	\item[(b)] There is a monochromatic-$H$-free $3$-edge-colouring of $G$ and $\delta = \delta(H) > 0$ such that if $q = o(n^{-1/m(H) + \delta})$ the colouring can with high probability be extended to a colouring of $G \cup G_{n,q}$ without monochromatic copies of $H$.
\end{itemize}
\end{theorem}

\begin{proof}
Since $F$ is $2$-balanced and contains a cycle, we have $d(F) \ge 1$ and $d_2(F) > 1$.  Using Lemma~\ref{lem:productdensity}(b), $m_2(H) = m_2(F)$.  Hence, if the constant $c$ is sufficiently small, the R\"odl--Ruci\'nski theorem implies that with high probability we can find a $2$-colouring $\varphi$ of $E(G)$ without any monochromatic copy of $F$.  This is the edge-colouring we extend in both cases.

\begin{itemize}
	\item[(a)]  In this case, extend $\varphi$ to the edges of $G_{n,q}$ arbitrarily.  Observe that $H = F \rprod{1} (F,f)$ consists of $e(F)$ edge-disjoint copies of $F$.  Since there are no monochromatic copies of $F$ in $G$, any monochromatic copy of $H$ in $G \cup G_{n,q}$ must contain at least $e(F)$ edges from $G_{n,q}$.

There are at most $n^{v(H)}$ potential copies of $H$ and $2^{e(H)}$ ways to distribute its edges between $G$ and $G_{n,q}$.  Since $q \le p$, the probability that a copy with at least $e(F)$ edges from $G_{n,q}$ appears in $G \cup G_{n,q}$ is at most $p^{e(H) - e(F)} q^{e(F)}$.  Thus, by the union bound, the probability that there is a copy of $H$ in $G \cup G_{n,q}$ with at least $e(F)$ edges from $G_{n,q}$ is at most
\[ n^{v(H)} 2^{e(H)} p^{e(H) - e(F)} q^{e(F)} = 2^{e(H)}c^{e(H) - e(F)} n^{v(F) + e(F)(v(F)-2) - e(F) (e(F)-1) / m_2(H)} q^{e(F)}, \]
where we used that $v(H) = v(F) + e(F)(v(F)-2)$ and $e(H) = e(F)^2$.
As $m_2(H) = m_2(F) = \frac{e(F)-1}{v(F)-2}$, this simplifies to $2^{e(H)}c^{e(H) - e(F)} n^{v(F)} q^{e(F)}$, which is $o(1)$ by our choice of $q$.  Hence, with high probability our arbitrary extension of $\varphi$ to $G \cup G_{n,q}$ does not create a monochromatic copy of $H$.
	
	\item[(b)] The colouring $\varphi$ uses two colours, say red and blue.  This leaves us with one unused colour, say green, that we can use when extending $\varphi$ to the edges of $G_{n,q}$.
	
	As $F$ is $2$-balanced, Lemma~\ref{lem:2become1} implies that it is also strictly balanced.  By Lemma~\ref{lem:productdensity}(a), it follows that $H$ is strictly balanced. As a consequence, any union of two copies of $H$ that share at least an edge must be strictly denser than $H$ itself.  Indeed, the subgraph common to both copies of $H$ is a proper subgraph and therefore strictly sparser than $H$.  Hence, the vertices and edges added in the second copy in the union must increase the overall density.
	
	There is thus some $\delta_1 = \delta_1(H) > 0$ such that, whenever $q = o(n^{-1/m(H) + \delta_1})$, intersecting copies of $H$ do not appear in $G_{n,q}$. That is,
	the copies of $H$ appearing in $G_{n,q}$ are with high probability pairwise edge-disjoint.  We shall choose our $\delta(H)$ to be less than this $\delta_1(H)$.
	
	We now order the edges of $G_{n,q}$ arbitrarily and process them one-by-one.  We colour each edge green, unless that would create a green copy of $H$, in which case we colour the edge red.  When colouring in this fashion, if we create a monochromatic copy of $H$, it clearly must be red.
	
	Consider a red copy $H_0$ of $H$ in our colouring of $G \cup G_{n,q}$.  Since $H_0$ is an edge-disjoint union of $e(F)$ copies of $F$ and the colouring $\varphi$ of $G$ has no monochromatic copy of $F$, each copy of $F$ in $H_0$ must contain at least one red edge from $G_{n,q}$.  An edge $e$ from $G_{n,q}$ is only red if it is the last edge of an otherwise green copy $H_e$ of $H$, which must be wholly contained in $G_{n,q}$.  Moreover, for $e \neq e'$, the copies $H_e$ and $H_{e'}$ of $H$ are edge-disjoint.
	
	This gives us a subgraph of $G \cup G_{n,q}$ with at most $v(F) + e(F)(v(F) - 2 + v(H) - 2)$ vertices and $e(F) ( e(F) + e(H) - 1)$ edges, of which at least $e(F)e(H)$ edges come from $G_{n,q}$. Since $q \le p$ and there are at most some constant $K$ ways of building such a subgraph and dividing its edges between $G$ and $G_{n,q}$, the probability of finding such a structure is at most 
	\[ K n^{v(F) + e(F) (v(F) - 2 + v(H) - 2)} p^{e(F) (e(F) - 1)} q^{e(F) e(H)}. \]
	
	Since $p \le n^{-1/m_2(H)} = n^{-(v(F)-2)/(e(F)-1)}$, this is at most $K n^{v(F) + e(F)(v(H) - 2)} q^{e(F)e(H)}$.  Now $q = o(n^{-1/m(H) + \delta}) = o(n^{-v(H) / e(H) + \delta})$, since $H$ is strictly balanced.  Thus the upper bound on the probability of the appearance of a red copy of $H$ is $o(n^{v(F) - 2e(F) + \delta e(F) e(H)}) = o(n^{e(F)(\delta e(H) - 1)})$, since $e(F) \ge v(F)$.  Hence, if we choose $\delta \le \min \{ 1/e(H), \delta_1(H) \}$, this probability is $o(1)$, so with high probability we can extend the colouring to the edges of $G_{n,q}$ without creating a monochromatic copy of $H$. \qedhere
\end{itemize}
\end{proof}

\section{The proof of Theorem~\ref{thm:main}}

Having shown in the previous section that some condition on the graph $H$ is necessary in Theorem~\ref{thm:main}, we now show that our condition is sufficient.  We begin with a sketch of the proof and then recall several useful results before providing the details of the argument.

\subsection{An overview of the proof}

We shall assume the colours used are red, blue and, in the case of three-colourings, green. Our goal is to find structures in the first random graph, $G$, that force the creation of a monochromatic copy of $H$ no matter how the edges of the second random graph, $G_{n,q}$, are coloured. To that end, we make the following definitions.

\begin{definition}[Colour-forced edges]
A copy of $H \setminus h$ in $G$ is \emph{supported} on the pair $\{x,y\}$ if $\{x,y\}$ maps to the missing edge $h$. We then call $\{x,y\}$ the \emph{base} of the copy.  Given an edge-colouring $\varphi$, we say $\{x,y\}$ is a \emph{red, blue or green base} if it is the base of a monochromatic copy of $H \setminus h$ of the corresponding colour.  Finally, we say a pair $\{x,y\}$ is \emph{green-forced} if it is both a red and a blue base simultaneously, with \emph{blue-forced} and \emph{red-forced} defined similarly.
\end{definition}

In the two-colour case, observe that it is impossible to extend $\varphi_2$ to a green-forced pair, since colouring it either red or blue would create a monochromatic copy of $H$.  For the first assertion of Theorem~\ref{thm:main}, we shall show that with high probability $G$ is such that every two-colouring $\varphi_2$ admits quadratically many green-forced pairs.  Then, again with high probability when $q_2 = \omega(n^{-2})$, one of these pairs will be an edge of the second random graph $G_{n,q_2}$, so any extension of $\varphi_2$ to $G \cup G_{n,q_2}$ will create a monochromatic copy of $H$.

When dealing with three colours, our goal will instead be to show that there is some colour, say green, such that the green-forced pairs in $G$ are sufficiently dense that, when $q_3 = \omega(n^{-1/m(H)})$, we will find a copy of $H$ in $G_{n,q}$ consisting solely of green-forced pairs.  If any one of its edges is coloured red or blue, it will complete a monochromatic copy of $H$ with edges from $G$.  On the other hand, if all of its edges are coloured green, we obtain a green copy of $H$ instead.

To find these colour-forced structures, we consider the reduced graph of a regular partition of $G$ (with respect to the colouring $\varphi_2$ or $\varphi_3$).  In this reduced graph we will find two colours, say red and blue, and a copy of $H \rrprod{2} (H,h)$ such that for each (removed) edge from the central copy of $H$, one of the attached copies of $(H,h)$ is monochromatic red and the other is monochromatic blue.  By applying the sparse counting lemma, we will deduce the existence of many potential copies of $H$ consisting of green-forced edges, from which we will be able to draw the desired conclusion.

Although the proof can be simplified in the two-coloured setting, for the sake of brevity we shall present a single unified argument allowing for three colours throughout, and only differentiate between the two cases at the end of the proof.

\subsection{Some preliminaries}

Here we collect several results about random graphs and sparse regularity that we shall use in our proof.

\subsubsection{Random graphs}

The Hoeffding inequality shows that $G_{n,p}$ does not have any subgraphs that are far sparser or denser than expected with high probability.

\begin{proposition} \label{prop:upperuniform}
Let $\eta > 0$ be fixed and suppose $p = \omega(n^{-1})$.  Then, with high probability, $G_{n,p}$ is such that the following holds for any disjoint sets $X, Y$ of vertices with $\abs{X}, \abs{Y} \ge \eta n$:
\begin{itemize}
	\item[(i)] $\tfrac12 \binom{\abs{X}}{2} p \le e(G_{n,p}[X]) \le 2 \binom{\abs{X}}{2} p$ and
	\item[(ii)] $\tfrac12 \abs{X} \abs{Y} p \le e(G_{n,p}[X,Y]) \le 2 \abs{X} \abs{Y} p$.
\end{itemize}
\end{proposition}

A simple application of Markov's inequality also shows that $G_{n,p}$ is unlikely to contain many  more copies of any subgraph than expected.

\begin{proposition} \label{prop:fewsubgraphs}
Given any graph $F$ with $v$ vertices and $e$ edges and any $K > 1$, the probability that there are more than $K n^v p^e$ copies of $F$ in $G_{n,p}$ is at most $1/K$.
\end{proposition}

In the other direction, we can use Chebyshev's inequality to establish the existence of subgraphs in $G_{n,p}$ when $p$ is suitably large. More precisely, it follows from Theorem 4.4.5 in The Probabilistic Method by Alon and Spencer~\cite{AS08} that if $p = \omega(n^{-1/m(F)})$, then the number of copies of $F$ in $G_{n,p}$ is concentrated around its expectation.

\begin{proposition} \label{prop:manysubgraphs}
Given a graph $F$ on $v$ vertices and a constant $\zeta > 0$, let $\mc F$ be a collection of $\zeta n^v$ potential copies of $F$.  If $p = \omega_n n^{-1/m(F)}$ with some $\omega_n=\omega(1)$, then the probability that $G_{n,p}$ does not contain a copy of $F$ from $\mc F$ is at most $\frac{v! 2^v}{\zeta \omega_n}$.
\end{proposition}

If the edge probability $p$ is even larger, then the following result, a consequence of Theorem~3.29 from the book Random Graphs by Janson, {\L}uczak and Ruci\'nski~\cite{JLR00}, shows that there will be many pairwise edge-disjoint copies of $H$ in $G_{n,p}$.

\begin{proposition} \label{prop:disjointcopies}
For every graph $H$ with $m_2(H) > 1$, there is a constant $\kappa = \kappa(H)$ such that, given constants $\rho, c > 0$ and setting $p = c n^{-1/m_2(H)}$, with high probability every induced subgraph of $G_{n,p}$ on at least $\tfrac12 \rho n$ vertices contains at least $\kappa c^{e(H) - 1} \rho^{v(H)} n^2 p$ edge-disjoint copies of $H$.
\end{proposition}

\subsubsection{Sparse regularity and counting}

Given an $n$-vertex graph $G$, two disjoint sets of vertices $X$ and $Y$ form an \emph{$(\eps,p)$-regular pair of density $d$} if $d(X,Y) = d$ and, for all $X' \subseteq X$ with $\abs{X'} \ge \eps \abs{X}$ and $Y' \subseteq Y$ with $\abs{Y'} \ge \eps \abs{Y}$, we have $\abs{d(X',Y') - d(X,Y)} < \eps p$,
where $d(U,V)$ denotes $\frac{e(U,V)}{|U||V|}$.  This notion of regularity is inherited by induced and random subgraphs (see Lemma 4.3 in~\cite{GS05}).

\begin{proposition} \label{prop:slicing}
Suppose that $c \in (0, \tfrac12]$, $G$ is a graph and $U, W$ are disjoint vertex sets, both of size $N$, with $(U,W)$ an $(\eps, p)$-regular pair of density $d = \omega(N^{-1})$.  Then the following is true:
\begin{itemize}
	\item[(i)] for $X \subset U$ and $Y \subset W$ with $\abs{X}, \abs{Y} \ge c N$, the pair $(X,Y)$ is $(\eps / c,p)$-regular with density at least $d - \eps p$ and
	\item[(ii)] for $m \ge cdN^2$, the subgraph $G'$ of $G$ obtained by choosing $m$ edges from $G[U,W]$ uniformly at random forms a $(2 \eps, p)$-regular pair with high probability.
\end{itemize}
\end{proposition}

An \emph{$(\eps, p)$-regular partition} $\mc P$ of $G$ is a partition $V(G) = V_0 \cup V_1 \cup \dots \cup V_k$ such that $\abs{V_0} \le \eps n$, $\abs{V_1} = \abs{V_2} = \dots = \abs{V_k}$ and all but at most $\eps k^2$ pairs $(V_i, V_j)$, $1 \le i < j \le k$, are $(\eps,p)$-regular.  When the graph $G$ is edge-coloured, we say a partition is \emph{$(\eps,p)$-regular} if for all but at most $\eps k^2$ pairs of parts the edges of each colour between the two parts form an $(\eps,p)$-regular subgraph. If $G$ has density $d$, we say it is \emph{$(\eta, D)$-upper-uniform} if, for all disjoint sets $X$ and $Y$ of size at least $\eta n$, we have $d(X,Y) \le D d$.  With these definitions in place, we may state a version of the sparse regularity lemma, originally due to Kohayakawa and R\"odl~\cite{K97}.

\begin{theorem} \label{thm:sparsereg}
For all $\eps, D > 0$ and $r, t \in \mathbb{N}$, there are $\eta > 0$ and $T \in \mathbb{N}$ such that every $r$-colouring of the edges of an $(\eta, D)$-upper-uniform graph $G$ of density $d$ on at least $T$ vertices has an $(\eps,d)$-regular partition $\mc P$ with $k$ parts for some $k \in [t,T]$.
\end{theorem}

The final ingredient we will need is a sparse counting lemma due to Conlon, Gowers, Samotij and Schacht~\cite{CGSS14}.  Given a graph $H$, integers $N$ and $m$, and $\eps, p > 0$, we define the family $\mc G(H,N,m,p,\eps)$ to be all graphs obtained by replacing each vertex of $H$ by an independent set of size $N$ and replacing each edge of $H$ by an $(\eps,p)$-regular bipartite graph with exactly $m$ edges.  Given such a graph $G$, let $G(H)$ denote the number of canonical copies of $H$ in $G$ (by which we mean that each vertex of $H$ in the copy belongs to the corresponding independent set in $G$).

\begin{theorem} \label{thm:sparsecount}
For every graph $H$ and every $d > 0$, there exist $\eps, \xi > 0$ with the following property.  For every $\eta > 0$, there is $C > 0$ such that if $p \ge C n^{-1/m_2(H)}$, then, with high probability, for every $N \ge \eta n$, $m \ge d p N^2$ and every subgraph $G$ of $G_{n,p}$ in $\mc G(H,N,m,p,\eps)$, $G(H) \ge \xi N^{v(H)} \left( \frac{m}{N^2} \right)^{e(H)}$.
\end{theorem}

\subsection{The reduced graph}

With these preliminaries in hand, we can proceed with the proof of Theorem~\ref{thm:main}.  We begin by describing the (standard) construction of the reduced graph and proving it has some useful properties.

Let $t, \eps, \alpha$ be defined such that $1/t \le \eps \ll \alpha \ll \kappa, c$, where $\kappa = \kappa(H)$ is the constant from Proposition~\ref{prop:disjointcopies}, $p = cn^{-1/m_2(H)}$ and `$\ll$' means these parameters are sufficiently small for the subsequent calculations to hold.

Now consider a monochromatic-$H$-free $3$-edge-colouring $\varphi$ of the edges of $G \sim G_{n,p}$ (where, as in the case of $\varphi_2$, we may only be using two of the three colours) and let $\gred$, $\gblue$ and $\ggreen$ represent the red, blue and green subgraphs of $G$, respectively. Given our choice of $\eps$ and $t$ and setting $r = 3$ and $D = 4$, let $\eta$ and $T$ be as in Theorem~\ref{thm:sparsereg}.  Proposition~\ref{prop:upperuniform} shows that $G$ is with high probability $(\eta, 4)$-upper-uniform.  Hence, there is an $(\eps,p)$-regular partition $V(G) = V_0 \cup V_1 \cup \hdots \cup V_k$, where $t \le k \le T$.

We next define three graphs, $\Gred$, $\Gblue$ and $\Ggreen$, on the same vertex set $[k]$.  $\Gred$ has an edge between $i$ and $j$ if and only if the bipartite induced subgraph $\gred[V_i,V_j]$ forms an $(\eps,p)$-regular pair of density at least $\alpha p$, with $\Gblue$ and $\Ggreen$ defined similarly with respect to $\gblue$ and $\ggreen$, respectively.  The reduced (multi)graph $\G$ is the coloured union of $\Gred$ in red, $\Gblue$ in blue and $\Ggreen$ in green.  Given a vertex $i \in [k]$, we write $N_{red}(i)$, $N_{blue}(i)$ and $N_{green}(i)$ for its neighbourhoods in $\Gred$, $\Gblue$ and $\Ggreen$, respectively, and write $d_{red}(i), d_{blue}(i)$ and $d_{green}(i)$ for the sizes of these sets.

\medskip

We first show that any induced subgraph of $\G$ with linearly many vertices has a vertex with large degree in at least two of the colours.

\begin{lemma} \label{lem:largedegrees}
Define $f(\rho) = \tfrac{1}{24} \kappa c^{e(H)-1} \rho^{v(H) - 1}$.  Suppose $\rho$ satisfies
\begin{equation} \label{ineq:rho}
	6 \rho f(\rho) \ge 3 \eps + \tfrac12 \alpha.
\end{equation}
Then, with high probability, for any subset $U \subseteq [k]$ of $\rho k$ vertices of the reduced graph $\Gamma$, we can find a vertex $u \in U$, two disjoint sets $X_1, X_2 \subset U$ of size at least $f(\rho) k$ and two distinct colours $\chi_1, \chi_2$ such that, for each $i \in [2]$, $u$ is adjacent to all vertices in $X_i$ with edges of colour $\chi_i$.
\end{lemma}

\begin{proof}
Let $W = \cup_{u \in U} V_u$ be the vertices in the parts of $G$ corresponding to the vertices of $U$ and note that $\abs{W} \ge (1 - \eps)\rho n \ge \tfrac12 \rho n$.  Hence, by Proposition~\ref{prop:disjointcopies}, we may assume $G[W]$ contains at least $24 \rho f(\rho) n^2 p$ edge-disjoint copies of $H$.  Since there are no monochromatic copies of $H$ in the $3$-edge-colouring of $G$, each such copy must contain two edges of distinct colours.  It easily follows that there are two colours, say red and blue, that each appear on at least $12 \rho f(\rho) n^2 p$ edges of $G[W]$.

Using Proposition~\ref{prop:upperuniform}, we observe that all but at most $(3 \eps + \tfrac12 \alpha) n^2 p$ red edges of $G[W]$ are contained within dense $(\eps,p)$-regular pairs.  Indeed, at most $k \cdot 2 \binom{n/k}{2} p \leq n^2 p/k \leq \eps n^2 p$ edges can be contained within the parts $V_i$, at most $\eps k^2 \cdot 2 (n/k)^2 p \leq 2 \eps n^2 p$ edges can be within irregular pairs $(V_i, V_j)$ and at most $\binom{k}{2} \cdot (n/k)^2 \alpha p \leq \frac{1}{2} \alpha n^2 p$ edges are within $(\eps,p)$-regular pairs $(V_i, V_j)$ of density less than~$\alpha p$. From~\eqref{ineq:rho}, it follows that there are at least $6 \rho f(\rho) n^2 p$ red edges in $G[W]$ that are contained in $(\eps,p)$-regular pairs of density at least $\alpha p$.  Again by Proposition~\ref{prop:upperuniform}, each such pair can account for at most $2 (n/k)^2 p$ edges in $G[W]$, so there must be at least $3 \rho f(\rho) k^2$ such pairs, each of which corresponds to an edge of $\Gred[U]$.  By symmetry, we also find at least $3 \rho f(\rho) k^2$ edges in $\Gblue[U]$.

Now let $A = \{ a \in U : d_{red}(a,U) \ge 2 f(\rho) k \}$.  By summing the red degrees of vertices in $U$, distinguishing between those in $A$ and those not, we have 
\[ 6 \rho f(\rho) k^2 \le \rho k \cdot \abs{A} + 2 f(\rho) k \cdot \rho k, \]
from which we deduce that $\abs{A} \ge 4 f(\rho) k$.  Defining $B = \{ b \in U : d_{blue}(b,U) \ge 2 f(\rho) k \}$, we similarly have $\abs{B} \ge 4 f(\rho) k$. If $A \cap B \neq \emptyset$, let $u \in A \cap B$.  Since $d_{red}(u,U), d_{blue}(u,U) \ge 2 f(\rho) k$, we can find the required disjoint sets $X_1$ and $X_2$ of size $f(\rho) k$ of red and blue neighbours, respectively.

Otherwise, for every $a \in A$ and $b \in B$, by Proposition~\ref{prop:upperuniform}, there are at least $\tfrac12 (n/k)^2 p$ edges in $G$ between $V_a$ and $V_b$, so one of the three colours appears on at least $\tfrac16 (n/k)^2 p > \alpha (n/k)^2 p$ edges.  Let $\chi_1$ be the colour that appears most commonly as the majority colour in these $\abs{A} \abs{B}$ pairs.  Ignoring the pairs that give rise to irregular pairs in $\Gamma_{\chi_1}$, it follows that there are at least $\tfrac13 \abs{A} \abs{B} - \eps k^2$ edges in $\Gamma_{\chi_1}$ between $A$ and $B$. Provided $\alpha$ is sufficiently large with respect to $\eps$, \eqref{ineq:rho} and our lower bound on $\abs{A}, \abs{B}$ imply this is at least $\tfrac14 \abs{A} \abs{B}$ edges.

If $\chi_1$ is not red, then take $\chi_2$ to be red and, by averaging, find some $u \in A$ with a set $X_1$ of at least $\tfrac14 \abs{B} \ge f(\rho) k$ neighbours in $B$ in the colour $\chi_1$.  Since $u \in A$, we have $d_{red}(u,U) \ge 2 f(\rho) k$, so we can find a disjoint set $X_2$ of $f(\rho) k$ red neighbours of $u$, as required.  Otherwise, if $\chi_1$ is red, we take $\chi_2$ to be blue.  By the same argument, we can find some $u \in B$ with a set $X_1$ of at least $f(\rho) k$ red neighbours in $A$ and, since $u \in B$, it has large enough degree in $\Gblue[U]$ to guarantee a disjoint set $X_2$ of blue neighbours.
\end{proof}

Through repeated use of this lemma, we can build large multicoloured structures in $\Gamma$.

\begin{corollary} \label{cor:twocliques}
Given $t \in \mathbb{N}$, let $\rho_0 = 1$ and, for $1 \le i \le 2t-2$, let $\rho_i = f(\rho_{i-1})$.  Provided $6 \rho_{2t-3} f(\rho_{2t-3}) \ge 3 \eps + \tfrac12 \alpha$, there is with high probability a vertex $v_0$ of $\Gamma$ contained in two monochromatic $t$-cliques of distinct colours.
\end{corollary}

\begin{proof}
Applying Lemma~\ref{lem:largedegrees} with $\rho = \rho_0 = 1$ and $U = [k]$, we find a vertex $u_0$ with large degrees in two colours. Without loss of generality, let the colours be red and blue.  In the first stage of this algorithm, we iterate within the red neighbourhood of $u_0$, finding either a vertex in blue and green $t$-cliques, in which case we are done, or a red $t$-clique containing $u_0$.

To start, apply Lemma~\ref{lem:largedegrees} again, this time taking $U$ to be the set of red neighbours of $u_0$.  This gives a vertex with large degrees in two colours.  While one of those colours is red, we repeat the process, giving us a sequence of vertices with large nested red neighbourhoods.  If this sequence (including $u_0$) has length $t-1$, by choosing an arbitrary vertex in the final red neighbourhood, we obtain a red $t$-clique containing $u_0$ and can proceed to the second stage.

Otherwise, after some $h \le t-2$ steps we obtain a vertex $u_1$ that has large blue and green neighbourhoods.  In this case, we first iterate within the blue neighbourhood of $u_1$.  Each subsequent vertex has either a large red neighbourhood or a large blue neighbourhood within which we can proceed.  Once we have obtained a sequence of $2t-3$ vertices, there are either $t-1$ of them (including $u_0$) for which we iterated within a red neighbourhood or $t-1$ of them (including $u_1$) for which we iterated within a blue neighbourhood.  In the first case, we choose an arbitrary vertex in the final neighbourhood to create a red $t$-clique containing $u_0$ and can then proceed to the second stage.

In the second case, choosing an arbitrary vertex in the final neighbourhood gives a blue $t$-clique containing $u_1$.  We can then return to the green neighbourhood of $u_1$ and repeatedly iterate, at each point proceeding with a red or green neighbourhood of the latest vertex.  Once we reach a sequence of length $2t-3$ (including the vertices between $u_0$ and $u_1$), we again either have $t-1$ vertices with green neighbourhoods or $t-1$ vertices with red neighbourhoods.  In the first case, we can complete a green $t$-clique containing $u_1$ that, together with the earlier blue $t$-clique, completes the desired structure.  In the second case, choosing a vertex in the final neighbourhood again completes a red $t$-clique containing $u_0$, with which we proceed to the second stage.

If we proceed to the second stage, we will have already found a red $t$-clique containing $u_0$. The second stage consists of mirroring the above process in the blue neighbourhood of $u_0$. This results in a blue $t$-clique containing $u_0$ or a vertex $u_2$ in the blue neighbourhood that is contained in both red and green $t$-cliques; in either case, we are done.
\end{proof}

\subsection{Building colour-forced structures}

Let $\mc K_n(H)$ denote the family of all copies of $H$ in $K_n$.  Using the cliques from Corollary~\ref{cor:twocliques}, we will prove the following key proposition.

\begin{proposition} \label{prop:colourforced}
There are positive constants $\kappa = \kappa(c,H)$ and $\zeta = \zeta(c,H)$ such that, for any $K > 1$, with probability at least $1 - \kappa K^{-1} - o(1)$, for every monochromatic-$H$-free $3$-edge-colouring $\varphi$ of $G$, there is some colour $\chi$ with at least $\zeta K^{-1} n^{v(H)}$ $\chi$-forced copies of $H$ in $\mc K_n(H)$.
\end{proposition}

\begin{proof}
For convenience, we write $v = v(H)$ and $e = e(H)$. Setting $t = e (v - 2) + 1$ and applying Corollary~\ref{cor:twocliques}, which holds with high probability, we find a vertex $x$ in the reduced graph $\Gamma$ that is in, say, both a red and a blue $K_t$.  Let $u_1, u_2, \hdots, u_{t-1}$ be the other vertices from the red clique and $w_1, w_2, \hdots, w_{t-1}$ be the other vertices from the blue clique.  Consider the corresponding parts in the graph $G$.  We know that, for all $1 \le i < j \le t-1$, the pairs $\gred[V_x, V_{u_i}], \gred[V_{u_i}, V_{u_j}], \gblue[V_x, V_{w_i}]$ and $\gblue[V_{w_i}, V_{w_j}]$ are all $(\eps, p)$-regular pairs of density at least $\alpha p$. This situation is illustrated below in the case $H = K_3$.

\begin{figure}[H] \label{fig:biclique}
    \centering

    \begin{tikzpicture}
        \foreach \a/\b in {120/180, 180/240, 240/120}
        {
            \draw[color=red!50, line width=0.3cm] (0:0)--(\a:2);
            \draw[color=red!50, line width=0.3cm] (\a:2)--(\b:2);
        }
        \foreach \a/\b in {300/0, 0/60, 60/300}
        {
            \draw[color=blue!50, line width=0.3cm] (0:0)--(\a:2);
            \draw[color=blue!50, line width=0.3cm] (\a:2)--(\b:2);
        }
        \draw[fill=white] (0:0) circle (0.3cm);
        \draw[color=red!50, fill=white] (120:2) circle (0.3cm);
        \draw[color=red!50, fill=white] (180:2) circle (0.3cm);
        \draw[color=red!50,fill=white] (240:2) circle (0.3cm);
        \draw[color=blue!50, fill=white] (300:2) circle (0.3cm);
        \draw[color=blue!50, fill=white] (0:2) circle (0.3cm);
        \draw[color=blue!50, fill=white] (60:2) circle (0.3cm);
    \end{tikzpicture}

    \caption{The parts of $G$ corresponding to the two cliques from Corollary~\ref{cor:twocliques}.}
\end{figure}
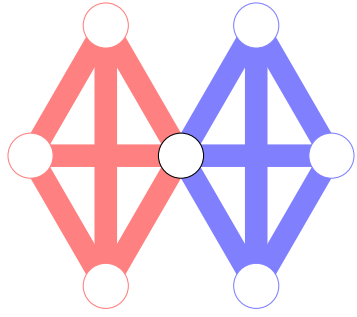

Partition the part $V_x$ into $v$ equal-sized subsets, $X_1, X_2, \hdots, X_v$, letting $N$ denote the size of these sets.  Define $\eta$ by $N = \eta n$, noting that $\eta \ge \frac{1-\eps}{kv}$, where we recall that $k \le T$ is the number of parts in the $(\eps,p)$-regular partition of $G$.  For each $i$, let $R_i \subset V_{u_i}$ and $B_i \subset V_{w_i}$ be arbitrary subsets of size $N$.  Let $\mc X = \{X_1, \hdots, X_v\}$, $\mc R = \{R_1, \hdots, R_{t-1} \}$ and $\mc B = \{ B_1, \hdots, B_{t-1} \}$.  By Proposition~\ref{prop:slicing}(i), it follows that the pairs $\gred[X_i, R_j], \gred[R_i, R_j], \gblue[X_i, B_j]$ and $\gblue[B_i, B_j]$ are all $(\eps v, p)$-regular of density at least $(\alpha - \eps)p$.

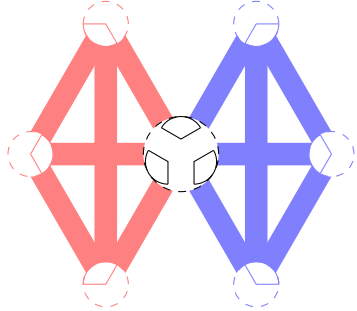
\begin{figure}[H] \label{fig:smallparts}
    \centering

    \begin{tikzpicture}
        \foreach \a/\b in {120/180, 180/240, 240/120}
        {
            \draw[color=red!50, line width=0.3cm] (0:0)--(\a:2);
            \draw[color=red!50, line width=0.3cm] (\a:2)--(\b:2);
        }
        \foreach \a/\b in {300/0, 0/60, 60/300}
        {
            \draw[color=blue!50, line width=0.3cm] (0:0)--(\a:2);
            \draw[color=blue!50, line width=0.3cm] (\a:2)--(\b:2);
        }
        
        \draw[color=red!50, style = dashed, fill = white] (120:2) circle (0.3cm);
        \begin{scope}[shift={(120:2)}]
            \draw[color=red!50] (180:0.3)--(0:0)--(300:0.3);
            \draw[color=red!50] (180:0.3) arc (180:300:0.3);
        \end{scope}
        
        \draw[color = red!50, style = dashed, fill = white] (180:2) circle (0.3cm);
        \begin{scope}[shift={(180:2)}]
            \draw[color = red!50] (300:0.3)--(0:0)--(60:0.3);
            \draw[color = red!50] (300:0.3) arc (300:420:0.3);
        \end{scope}
        
        \draw[color = red!50, style = dashed, fill = white] (240:2) circle (0.3cm);
        \begin{scope}[shift={(240:2)}]
            \draw[color = red!50] (60:0.3)--(0:0)--(180:0.3);
            \draw[color = red!50] (60:0.3) arc (60:180:0.3);
        \end{scope}
        
        \draw[color = blue!50, style = dashed, fill = white] (300:2) circle (0.3cm);
        \begin{scope}[shift={(300:2)}]
            \draw[color = blue!50] (0:0.3)--(0:0)--(120:0.3);
            \draw[color = blue!50] (0:0.3) arc (0:120:0.3);
        \end{scope}
        
        \draw[color = blue!50, style = dashed, fill = white] (0:2) circle (0.3cm);
        \begin{scope}[shift={(0:2)}]
            \draw[color = blue!50] (120:0.3)--(0:0)--(240:0.3);
            \draw[color = blue!50] (120:0.3) arc (120:240:0.3);
        \end{scope}
        
        \draw[color = blue!50, style = dashed, fill = white] (60:2) circle (0.3cm);
        \begin{scope}[shift={(60:2)}]
            \draw[color = blue!50] (240:0.3)--(0:0)--(0:0.3);
            \draw[color = blue!50] (240:0.3) arc (240:360:0.3);
        \end{scope}
        
        \draw[style = dashed, fill = white] (0:0) circle (0.5cm);
    
        \begin{scope}[shift={(90:0.2)}]
            \draw (30:0.3) -- (0:0) -- (150:0.3);
            \draw (30:0.3) arc (30:150:0.3);
        \end{scope}
        
        \begin{scope}[shift={(210:0.2)}]
            \draw (150:0.3) -- (0:0) -- (270:0.3);
            \draw (150:0.3) arc (150:270:0.3);
        \end{scope}
        
        \begin{scope}[shift={(330:0.2)}]
            \draw (270:0.3) -- (0:0) -- (30:0.3);
            \draw (270:0.3) arc (270:390:0.3);
        \end{scope}
    \end{tikzpicture}

    \caption{We divide the central part into $v(H)$ subsets and shrink the other parts accordingly.}
\end{figure}

Next consider the graph $H \rrprod{2} (H,h)$ and note that it has precisely $v + 2(t-1)$ vertices, with one central copy $H_0$ of $H$, whose edges are deleted, and each deleted edge $g \in E(H_0)$ supporting two otherwise vertex-disjoint copies $H_{g,1}$ and $H_{g,2}$ of $H$.  We can build a bijection $\psi : V(H \rrprod{2} (H,h)) \rightarrow \mc X \cup \mc R \cup \mc B$ such that:
\begin{itemize}
	\item $\psi(H_0) = \mc X$ and
	\item for all $g \in E(H_0)$, $\psi(V(H_{g,1}) \setminus g) \subset \mc R$ and $\psi(V(H_{g,2}) \setminus g) \subset \mc B$.
\end{itemize}
That is, for each edge $g \in E(H_0)$, we send one of the attached copies of $H$ to the red parts $\mc R$ and the other copy to the blue parts $\mc B$.

\begin{figure}[H] \label{fig:forcedparts}
    \centering

    \begin{tikzpicture}
        \foreach \a/\b in {60/90,60/330,180/90,180/210,300/210,300/330}
            \draw[color = red!50, line width = 0.2cm] (\a:2) -- (\b:0.5);
            
        \foreach \a/\b in {0/90,0/330,120/90,120/210,240/210,240/330}
            \draw[color = blue!50, line width = 0.2cm] (\a:2) -- (\b:0.5);
            
        \foreach \a in {90,210,330}
            \draw[color = white, fill = white] (\a:0.5) circle (0.3cm);
            
        \foreach \a in {0,60,120,180,240,300}
            \draw[color = white, fill = white] (\a:2) circle (0.3cm);
            
        \begin{scope}[shift={(150:0.15)}]
            \draw[color = gray!25, line width = 0.3cm, style = dashed] (90:0.5) -- (210:0.5);
        \end{scope}
        
        \begin{scope}[shift={(270:0.15)}]
            \draw[color = gray!25, line width = 0.3cm, style = dashed] (210:0.5) -- (330:0.5);
        \end{scope}
        
        \begin{scope}[shift={(30:0.15)}]
            \draw[color = gray!25, line width = 0.3cm, style = dashed] (330:0.5) -- (90:0.5);
        \end{scope}
        
        \begin{scope}[shift={(60:2)}]
            \draw[color = red!50] (180:0.3) -- (0:0) -- (300:0.3);
            \draw[color = red!50] (180:0.3) arc (180:300:0.3);
        \end{scope}
        
        \begin{scope}[shift={(180:2)}]
            \draw[color = red!50] (300:0.3) -- (0:0) -- (60:0.3);
            \draw[color = red!50] (300:0.3) arc (300:420:0.3);
        \end{scope}
        
        \begin{scope}[shift={(300:2)}]
            \draw[color = red!50] (60:0.3) -- (0:0) -- (180:0.3);
            \draw[color = red!50] (60:0.3) arc (60:180:0.3);
        \end{scope}
        
        \begin{scope}[shift={(0:2)}]
            \draw[color = blue!50] (120:0.3) -- (0:0) -- (240:0.3);
            \draw[color = blue!50] (120:0.3) arc (120:240:0.3);
        \end{scope}
        
        \begin{scope}[shift={(120:2)}]
            \draw[color = blue!50] (240:0.3) -- (0:0) -- (0:0.3);
            \draw[color = blue!50] (240:0.3) arc (240:360:0.3);
        \end{scope}
        
        \begin{scope}[shift={(240:2)}]
            \draw[color = blue!50] (0:0.3) -- (0:0) -- (120:0.3);
            \draw[color = blue!50] (0:0.3) arc (0:120:0.3);
        \end{scope}
        
        \begin{scope}[shift={(90:0.5)}]
            \draw (30:0.3) -- (0:0) -- (150:0.3);
            \draw (30:0.3) arc (30:150:0.3);
        \end{scope}
        
        \begin{scope}[shift={(210:0.5)}]
            \draw (150:0.3) -- (0:0) -- (270:0.3);
            \draw (150:0.3) arc (150:270:0.3);
        \end{scope}
        
        \begin{scope}[shift={(330:0.5)}]
            \draw (270:0.3) -- (0:0) -- (30:0.3);
            \draw (270:0.3) arc (270:390:0.3);
        \end{scope}
    \end{tikzpicture}

    \caption{We imagine a copy of $H$ between the subsets of the central part, with each edge supporting both a red and a blue copy of $H \setminus h$ using the parts from the red and blue cliques.}
\end{figure}
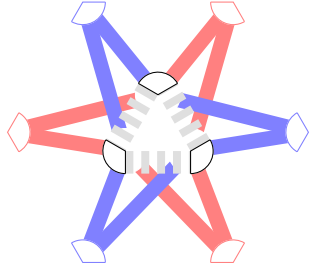

Let $m = \tfrac12 \alpha p N^2$ and consider an edge $f = \{y,z\} \in E(H \rrprod{2} (H,h))$.  If $f \in E(H_{g,1})$ for some $g \in E(H_0)$, then the pair $\gred[\psi(y), \psi(z)]$ is an $(\eps v, p)$-regular pair of density at least $(\alpha - \eps) p$.  Define $\psi(f) \subseteq \gred[\psi(y), \psi(z)]$ to be the subgraph obtained from this pair by choosing $m$ edges uniformly at random.  By Proposition~\ref{prop:slicing}, $\psi(f)$ is $(2 \eps v, p)$-regular with high probability. Otherwise, $f \in E(H_{g,2})$ for some $g \in E(H_0)$, in which case we define $\psi(f)$ to be the subgraph obtained by selecting $m$ edges uniformly at random from $\gblue[\psi(y), \psi(z)]$.  We again have, with high probability, that $\psi(f)$ is $(2 \eps v, p)$-regular.

Now define the subgraph $G' \subset G$ to be the union of all these subgraphs $\psi(f)$, that is, 
\[ V(G') = \bigcup_{y \in V(H \rrprod{2} (H,h))} \psi(y) \quad \textrm{and} \quad E(G') = \bigcup_{f \in E(H \rrprod{2} (H,h))} \psi(f). \]
From the above discussion, it is clear that $G' \in \mc G(H \rrprod{2} (H,h), N, m, p, 2 \eps v)$, where this family of graphs is as defined before Theorem~\ref{thm:sparsecount}.  Since $p = cn^{-1/m_2(H)}$ and, by Lemma~\ref{lem:productdensity}(c), $m_2(H \rrprod{2} (H,h)) < m_2(H)$, we can apply Theorem~\ref{thm:sparsecount}.  This gives some constant $\xi > 0$ such that there are with high probability at least $\xi N^{v(H \rrprod{2} (H,h))} \left( \tfrac{m}{N^2} \right)^{e(H \rrprod{2} (H,h))}$ copies of $H \rrprod{2} (H,h)$ in $G'$, where each vertex $y$ comes from the set $\psi(y)$.  To simplify this expression, we define $c' = \xi \eta^{v(H \rrprod{2} (H,h))} \left( \tfrac12 \alpha \right)^{e(H \rrprod{2} (H,h))}$ and $\mu = n^{v - 2} p^{e - 1}$.  Our lower bound on the number of copies of $H \rrprod{2} (H,h)$ can then be written as $c' \mu^{2e} n^v$.  Note that $c' > 0$ is a constant, while, since $p = cn^{-1/m_2(H)} \ge cn^{-(v-2)/(e-1)}$, $\mu = \Omega(1)$.  

In each such copy of $H \rrprod{2} (H,h)$, each missing edge $g \in E(H_0)$ in the central copy of $H$ supports both a red copy $H_{g,1}$ of $H \setminus h$ and a blue copy $H_{g,2}$ of $H \setminus h$.  In particular, this means $g$ is green-forced and, as this holds for all edges $g$, this shows that the central copy $H_0$ forms a green-forced copy of $H$ in $\mc K_n(H)$.

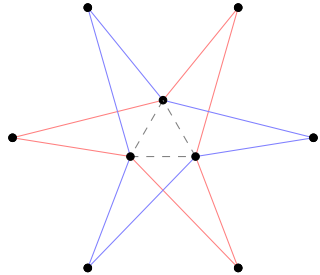
\begin{figure}[H]\label{fig:colorforcing}
    \centering
    
    \begin{tikzpicture}
        \foreach \a/\b in {60/90,60/330,180/90,180/210,300/210,300/330}
        {
            \draw[color = red!50] (\a:2) -- (\b:0.5);
            \draw[fill = black] (\a:2) circle (0.05cm);
        }
            
        \foreach \a/\b in {0/90,0/330,120/90,120/210,240/210,240/330}
        {
            \draw[color = blue!50] (\a:2) -- (\b:0.5);
            \draw[fill = black] (\a:2) circle (0.05cm);
        }
            
        \foreach \a/\b in {90/210,210/330,330/90}
        {
            \draw[color = gray, style = dashed] (\a:0.5) -- (\b:0.5);
            \draw[fill = black] (\a:0.5) circle (0.05cm);
        }
    \end{tikzpicture}
    
    \caption{Applying Theorem~\ref{thm:sparsecount} gives many copies of $H \rrprod{2} (H,h)$ in which the central copy of $H$ is green-forced.}
\end{figure}

However, we are not quite done, as these green-forced copies of $H$ may contribute to multiple copies of $H \rrprod{2} (H,h)$, in which case they will have been overcounted.  To rectify this, and complete the proof, we now show that most of these copies of $H$ are not counted too often.

To this end, suppose we have found $r$ distinct green-forced central copies $H_0$ of $H$ above and enumerate them as $H^{(1)}, H^{(2)}, \hdots, H^{(r)}$.  For each $1 \le i \le r$, let $Z_i$ denote the number of copies of $H \rrprod{2} (H,h)$ found above in which $H^{(i)}$ is the central copy $H_0$.  We have thus far established that
\begin{equation} \label{ineq:firstmoment}
	\sum_{i=1}^r Z_i \ge c' \mu^{2e} n^v,
\end{equation}
while we wish to show that $r \ge \tfrac{\zeta}{K} n^v$.

Now consider the quantity $Z_i^2$.  This counts the number of ordered pairs $(A,B)$ of canonical copies of $H \rrprod{2} (H,h)$ with $H^{(i)}$ as the central copy $H_0$.  Given such a pair, let $J = A \cup B$.  In $J$, each edge $g \in H_0$ is contained in a red copy of $H \setminus h$ from $A$ and one from $B$ as well.  The same holds true for the blue copies of $H \setminus h$. These attached copies of $H \setminus h$ in $J$ are mostly disjoint outside the central $H_0$, except that the two copies of the same colour supported on the same edge $g$ may share some vertices.  We consider such a graph $J$ as a degenerate copy of $H \rrprod{4} (H,h)$.  There are several isomorphism classes $J$ could belong to, depending on which vertices are shared by $A$ and $B$.

For each edge $g \in H_0$, let $F_{g,1}$ be the subgraph of $H$ induced by the vertices shared between the red copies of $H \setminus h$ in $A$ and $B$ supported on $g$ and define $F_{g,2}$ analogously for the blue copies.  We include the edge $g$ in $F_{g,1}$ and $F_{g,2}$, even though it does not appear in $J$.  Note that the union $\cup_{g \in E(H_0)} \cup_{j=1}^2 F_{g,j}$ determines the isomorphism class of $J$.  Hence, there are at most $2^{v(H \rrprod{2} (H,h))}$ possible isomorphism types, as for each vertex in $H \rrprod{2} (H,h)$, we can decide whether or not it belongs to the corresponding $F_{g,j}$.  Set $\kappa = 2^{v(H \rrprod{2} (H,h))}$.

We shall use Proposition~\ref{prop:fewsubgraphs} to show that, regardless of isomorphism type, there cannot be many copies of $J$ in $G$.  Indeed, we have 
\[ v(J) = v(H) + \sum_{g \in E(H_0)} \sum_{j=1}^2 \left(2 (v(H) - 2) - (v(F_{g,j}) - 2) \right) = v + 4e(v-2) - \sum_{g,j} (v(F_{g,j}) - 2) \]
and
\[ e(J) = \sum_{g \in E(H_0)} \sum_{j=1}^2 \left(2 (e(H) - 1) - (e(F_{g,j}) - 1) \right) = 4e (e-1) - \sum_{g,j} (e(F_{g,j}) - 1). \]
This gives
\[ n^{v(J)} p^{e(J)} = n^{v + 4e(v-2) - \sum_{g,j} (v(F_{g,j}) - 2)} p^{4e(e-1) - \sum_{g,j} (e(F_{g,j}) - 1)} = \frac{\mu^{4e} n^v}{\prod_{g,j} \left( n^{v(F_{g,j}) - 2} p^{e(F_{g,j}) - 1} \right) }. \]
Since $F_{g,j} \subseteq H$ and $p = cn^{-1/m_2(H)}$, we have $n^{v(F_{g,j}) - 2} p^{e(F_{g,j}) - 1} \ge c^{e(F_{g,j}) - 1}$ for all $g,j$.  Thus, $n^{v(J)} p^{e(J)} \le c^{-2e^2} \mu^{4e} n^v$.  Hence, by Proposition~\ref{prop:fewsubgraphs}, with probability at least $1 - K^{-1}$ there are at most $K c^{-2e^2} \mu^{4e} n^v $ copies of $J$ in $G$.  Taking a union bound over all isomorphism classes, we find that with probability at least $1 - \kappa K^{-1}$, there are at most $\kappa K c^{-2e^2} \mu^{4e} n^v$ of these degenerate copies of $H \rrprod{4} (H,h)$ in $G$.

We noted earlier that each pair $(A,B)$ of copies of $H \rrprod{2} (H,h)$ counted by $\sum_i Z_i^2$ gives rise to a degenerate copy $J = A \cup B$ of $H \rrprod{4} (H,h)$.  To reverse the correspondence, for each vertex in $H \rrprod{4} (H,h) \setminus (A \cap B)$, we must decide how to assign the corresponding vertices of $J$ to $A$ and $B$.  Thus, there are at most $\kappa^2 \geq 2^{v(H \rrprod{4} (H,h))}$ pairs $(A,B)$ giving rise to the same $J = A \cup B$.

Putting all this together, we have, with probability at least $1 - \kappa K^{-1}$,
\[ \sum_{i=1}^r Z_i^2 \le \kappa^3 K c^{-2e^2} \mu^{4e} n^v. \]

Define $I = \{ i : Z_i \ge \tfrac{2}{c'} \cdot \kappa^3 K c^{-2e^2} \mu^{2e} \}$.  It then follows from the above inequality that $\sum_{i \in I} Z_i \le \tfrac12 c' \mu^{2e} n^v$.  Plugging this into~\eqref{ineq:firstmoment}, we obtain $\sum_{i \notin I} Z_i \ge \tfrac12 c' \mu^{2e} n^v$.  As there are at most $r$ summands, each of which has size less than $\tfrac{2}{c'} \cdot \kappa^3 K c^{-2e^2} \mu^{2e}$, we can conclude that
\[ r \ge \left( \frac{(c')^2 c^{2e^2}}{4 \kappa^3 K} \right) n^v. \]
Setting $\zeta = \tfrac{1}{4} (c')^2 c^{2e^2} \kappa^{-3}$ completes the proof.
\end{proof}

\subsection{Finishing the proof}

We begin with part (a).  Suppose we have a monochromatic-$H$-free $2$-edge-colouring $\varphi_2$ of $G$ and $q_2 = \omega_n n^{-2}$ for some $\omega_n \rightarrow \infty$.  Set $K = \omega_n^{1/2}$.  By Proposition~\ref{prop:colourforced}, with probability $1 - \kappa K^{-1} - o(1) = 1 - o(1)$, there is some colour $\chi$ such that there are at least $\zeta K^{-1} n^{v(H)}$ $\chi$-forced copies of $H$.  As the colouring $\varphi_2$ only has red and blue edges, the colour $\chi$ must be green.

Each edge can be in at most $n^{v(H)-2}$ green-forced copies of $H$, so there must be at least $\zeta K^{-1} n^2$ green-forced edges.  If any of these edges were to appear in $G_{n,q_2}$, we would not be able to extend the colouring $\varphi_2$, as colouring the edge red or blue creates a monochromatic copy of $H$.  Hence, the probability that $\varphi_2$ extends to $G \cup G_{n,q_2}$ is at most
\[ \left( 1 - q_2 \right)^{\zeta K^{-1} n^2} \le \exp \left( - \zeta K^{-1} n^2 q_2 \right) = \exp \left( - \zeta K \right) = o(1), \]
as required.

Part (b) follows the same lines.  We begin as before: given the colouring $\varphi_3$ and some $q_3 = \omega_n n^{-1/m(H)}$, where $\omega_n \rightarrow \infty$, we set $K = \omega_n^{1/2}$.  By Proposition~\ref{prop:colourforced}, with probability $1 - o(1)$, there is some colour $\chi$ with at least $\zeta K^{-1} n^{v(H)}$ $\chi$-forced copies of $H$.

If any $\chi$-forced copy of $H$ appears in $G_{n,q_3}$, then $\varphi_3$ cannot be extended. Indeed, colouring all of its edges with the colour $\chi$ clearly creates a monochromatic copy of $H$, but since all the edges are $\chi$-forced, using any other colour on an edge also completes a monochromatic copy.  By Proposition~\ref{prop:manysubgraphs}, the probability that none of the $\chi$-forced copies of $H$ appear in $G_{n,q_3}$ is at most
\[ \frac{v(H)! 2^{v(H)} K}{\zeta \omega_n} = \frac{v(H)! 2^{v(H)}}{ \zeta K} = o(1), \]
as desired.  This completes the proof of Theorem~\ref{thm:main}.
 
\section{Concluding remarks}

Our investigations point to several open problems, perhaps the most interesting of which is to classify all graphs $H$ for which Theorem~\ref{thm:main} holds. We have shown that our condition, that there exists an edge $h$ such that $m_2(H \setminus h) < m_2(H)$, cannot be entirely dispensed with. However, there are also examples of graphs which do not satisfy this condition, but still satisfy some of the conclusions of Theorem~\ref{thm:main}.

Indeed, our proof of Theorem~\ref{thm:main}(a) readily generalises to the following statement.

\begin{theorem} \label{thm:general}
    Given a graph $H$, suppose there are graphs $\fred$, $\fblue$ and a matching $M$ such that
    \begin{itemize}
        \item[(i)] $V(\fred) \cap V(\fblue) = V(M)$, with $V(M)$ forming an independent set in $\fred$ and $\fblue$,
        \item[(ii)] $m_2(H) > m_2(\fred \cup \fblue)$,
        \item[(iii)] $m_2(H) \ge \frac{e(J)}{v(J) - v(M)}$ for all $J \subseteq \fred \cup \fblue$ with $V(M) \subset V(J)$ and $e(J) \ge 1$ and
        \item[(iv)] for any partition of the matching $M = \mred \cup \mblue$, $H$ is a subgraph of $\fred \cup \mred$ or $\fblue \cup \mblue$.
    \end{itemize}
    Let $c > 0$ be fixed and, for $p = c n^{-1/m_2(H)}$, let $G = G_{n,p}$.  Then, with high probability, the following holds.  Let $\varphi$ be an arbitrary monochromatic-$H$-free $2$-edge-colouring of $G$. If $q = \omega(n^{-2})$, then, with high probability, $\varphi$ cannot be extended to a monochromatic-$H$-free $2$-edge-colouring of $G \cup G_{n,q}$.
\end{theorem}

One of the simplest examples satisfying the conditions of Theorem~\ref{thm:general} is the graph $H$ consisting of two triangles joined by a path of length $\ell \geq 2$. In this case we can take $M$ to have size three with the corresponding graphs $\fred$ and $\fblue$ depicted below.

This example shows that the condition in Theorem~\ref{thm:main} is not best possible, as Theorem~\ref{thm:general} applies to a wider class of graphs.  However, there is a subtle trade-off in finding appropriate forcing structures $\fred \cup \fblue$ for Theorem~\ref{thm:general} --- we need them to be sparse enough to satisfy (ii) and (iii), but to have enough copies of $H$ for (iv).

\begin{figure}[H] \label{fig:blah}
    \centering
    \begin{tikzpicture}[scale=1.2,every node/.style={circle,draw,color=black,fill=black,inner sep=0pt,minimum width=3pt}]
       \foreach \i in {1,2}
        {
        \node (a\i) at (1-3*\i,-.35) {};
        \node (b\i) at (0-3*\i,-.35) {};
        \node (c\i) at (.5-3*\i,.35) {};
        
        \draw (a\i)--(b\i);
        \draw (b\i)--(c\i);
        \draw (c\i)--(a\i);
        }
        
        \draw (c2) to [out=-15,in=-165] (c1);
        \node[draw=none,fill=none] at (-4,.3) {$\ell$};
        
       \foreach \i in {0,1,2}
        {
         \node (u\i) at (0+3*\i,0) {};
         \node (v\i) at (1+3*\i,0) {};
         \node (x\i) at (0.5+3*\i,.7) {};
         \node (y\i) at (0.5+3*\i,-.7) {};
         
         \draw[dashed] (u\i)--(v\i);
         \draw[color=red] (u\i)--(x\i);
         \draw[color=red] (v\i)--(x\i);
         \draw[color=blue] (u\i)--(y\i);
         \draw[color=blue] (v\i)--(y\i);
         }
         
    \draw[color=red] (x0) to [out=-15,in=-165] (x1);
    \node[draw=none,fill=none] at (5,.65) {$\ell$};
    \draw[color=red] (x1) to [out=-15,in=-165] (x2);
    \node[draw=none,fill=none] at (2,.65) {$\ell$};
    \draw[color=blue] (y0) to [out=15,in=165] (y1);
    \node[draw=none,fill=none] at (5,-.65) {$\ell$};
    \draw[color=blue] (y1) to [out=15,in=165] (y2);
    \node[draw=none,fill=none] at (2,-.65) {$\ell$};

    \draw[color=red] (x0) to [out=30,in=150] (x2);
    \node[draw=none,fill=none] at (3.5,1.4) {$\ell$};
    \draw[color=blue] (y0) to [out=-30,in=-150](y2);
    \node[draw=none,fill=none] at (3.5,-1.4) {$\ell$};

    \end{tikzpicture}
    \caption{$H$, drawn on the left, consists of two triangles joined by a path of length $\ell$. On the right, $\fred$ is drawn in red, $\fblue$ in blue and the matching $M$ is drawn with dashed lines.}
\end{figure}
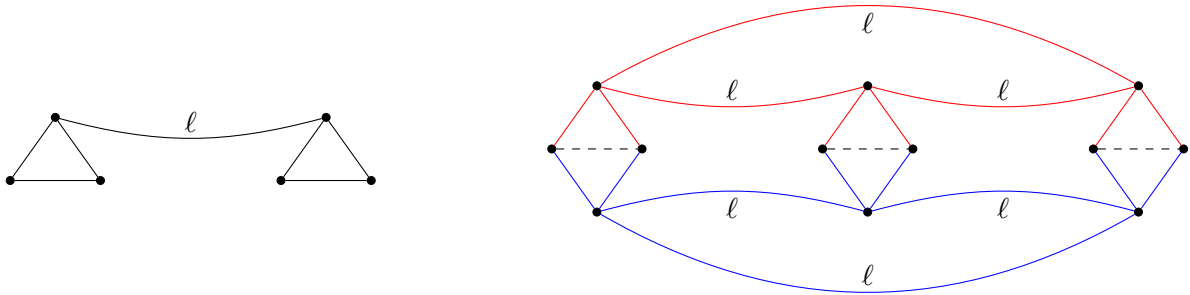

\medskip

\noindent
\textbf{Acknowledgements.} Part of this work was carried out while the third author visited the second and fourth authors at FU Berlin and he is grateful for their hospitality.

\end{document}